\documentclass[oneside,english]{amsart}
\usepackage[T1]{fontenc}
\usepackage[latin9]{inputenc}
\usepackage{color}
\usepackage{babel}
\usepackage{units}
\usepackage{amsthm}
\usepackage{amstext}
\usepackage{amssymb}
\usepackage{esint}
\usepackage[unicode=true,
 bookmarks=true,bookmarksnumbered=false,bookmarksopen=false,
 breaklinks=false,pdfborder={0 0 1},backref=false,colorlinks=true]
 {hyperref}
\hypersetup{pdftitle={A Note on Lipschitz Continuity of Solutions of Poisson Equations in Metric Measure Spaces},
 pdfauthor={Martin Kell}}

\makeatletter
\numberwithin{equation}{section}
\numberwithin{figure}{section}
  \theoremstyle{remark}
  \newtheorem*{rem*}{\protect\remarkname}
\theoremstyle{plain}
\newtheorem{thm}{\protect\theoremname}
  \theoremstyle{plain}
  \newtheorem{lem}[thm]{\protect\lemmaname}
  \theoremstyle{plain}
  \newtheorem{cor}[thm]{\protect\corollaryname}
  \theoremstyle{plain}
  \newtheorem{prop}[thm]{\protect\propositionname}

\makeatother

  \providecommand{\corollaryname}{Corollary}
  \providecommand{\lemmaname}{Lemma}
  \providecommand{\propositionname}{Proposition}
  \providecommand{\remarkname}{Remark}
\providecommand{\theoremname}{Theorem}

\begin{document}

\title[A Note on Lipschitz Continuity of Solutions of Poisson Equations]{A Note on Lipschitz Continuity of Solutions of Poisson Equations
in Metric Measure Spaces}

\author{Martin Kell}

\email{mkell@mis.mpg.de}

\address{Max-Planck-Institute for Mathematics in the Sciences, Inselstr. 22,
D-04103 Leipzig, Germany}

\thanks{The author wants to thank is advisor Prof. Jürgen Jost and the MPI-MiS
for providing an inspiring research environment and the IMPRS ``Mathematics
in the Sciences'' for financial support.}

\begin{abstract}
In this note we show how to adjust some proofs of Koskela et. al 2003
and Jiang 2011 in order to prove that in certain spaces $(X,d,\mu)$,
like $RCD(K,N)$-spaces, every Sobolev function with local $L^{p}$-Laplacian
and $p>\dim\mu$ is locally Lipschitz continuous.
\end{abstract}
\maketitle
In \cite{Koskela2003} and \cite{Jiang2011} Koskela et. al. and Jiang
showed that functions with local $L^{p}$-Laplacian, in particular
harmonic functions, are locally Lipschitz continuous if the space
is locally Ahlfors regular, a local uniform Poinca\'{e} inequality
holds and its heat flows admits a kind of Sobolev-Poincar\'{e} inequality.
In this note, we want to show that Ahlfors regularity can be replaced
by local uniform doubling and a weak upper bound on the volume growth.
In particular, every $RCD(K,N)$-space satisfies these conditions
(see \cite{Gigli2012} for nicely developed calculus and further references).
Since the proofs are almost the same as the ones in \cite{Jiang2011}
we will only show the necessary adjustments and refer to \cite{Koskela2003,Jiang2011}
for notation and precise statement.
\begin{rem*}
After this note was finished, we learnt that Jiang \cite{Jiang2013}
worked on an extension of \cite{Jiang2011}. Under similar assumptions,
but without the upper bound on the volume growth, he shows that functions
with local $L^{\infty}$-Laplacian are locally Lipschitz continuous
and additionally gives some gradient bounds for those functions. 
\end{rem*}
Throughout this note we assume that $(X,d,\mu)$ is a complete metric
measure space which is infinitesimal Hilbertian (see Gigli \cite{Gigli2012}).
For those spaces the Cheeger energy 
\[
\operatorname{Ch}(u)=\int|\nabla u|_{w}^{2}d\mu
\]
is a Dirichlet form, where $|\nabla u|_{w}$ is the minimal weak upper
gradient for $u$. We will drop the subscript $w$ and use the following
notation 
\[
\mathcal{E}(u,v)=\int\langle\nabla u,\nabla v\rangle d\mu
\]
such that 
\[
\mathcal{E}(u,u)=\operatorname{Ch}(u).
\]
Because $\operatorname{Ch}$ is a Dirichlet form the Sobolev space
\[
W^{1,2}(X)=\{u\in L^{2}(X)|\operatorname{Ch}(u)<\infty\}
\]
is a Hilbert space and for a dense subset of $W^{1,2}(X)$ we can
define the Laplacian $\Delta u\in L^{2}$ such that 
\[
\mathcal{E}(u,v)=-\int v\Delta ud\mu.
\]
Note that whereas $\langle\nabla u,\nabla v\rangle$ is a well-defined
object in $L^{1}$, $\nabla u$ is not defined. 
\begin{rem*}
Koskela et. al. \cite{Koskela2003} and Jiang \cite{Jiang2011} use
Cheeger derivatives, which are not necessary to develop the regularity
theory. Furthermore, it is not necessary to start with the Cheeger
energy. If we assume that a space is equipped with a Dirichlet form
$\mathcal{F}$ such that the induced metric $d_{\mathcal{F}}$ is
Lipschitz equivalent to the original metric, then the associated Laplacian
satisfies the Leibniz rule and a similar regularity theory could be
developed if doubling, Poincar\'{e} and Sobolev-Poincar\'{e} holds.
Furthermore, according to \cite{Koskela2012} if $\mathcal{F}$ satisfies
the Bakry-\'Emery condition then it is the Cheeger energy to $d_{\mathcal{F}}$
$ $ and the space $(X,d_{\mathcal{F}},\mu)$ is an $RCD(K,\infty)$-space.
However, if the Sobolev-Poincar\'{e} inequality holds uniformly independent
of $x$ then according to \cite{Bakry1997} the Bakry-\'Emery condition
holds.
\end{rem*}
We will assume the following doubling condition on the measure $\mu$:
for fixed $R>0$ and all $0<r<R$ there are constants $d_{R}>0$ and
$C_{R}>0$ such that 
\[
d_{R}r^{Q}\le\mu(B_{r}(x))
\]
 and 
\[
\mu(B_{2r}(x))\le C_{R}\mu(B_{r}(x)).
\]
Since we only want to show which adjustments are necessary, we will
further assume $Q\ge2$. This is no limitation on the spaces, since
every space satisfying the volume growth with $d_{R}$ and $Q$ will
satisfy it with $Q'\ge Q$ and some different $d_{R}^{'}$. We remark
that the case $Q\in(1,2)$ can be treated similar to \cite{Jiang2011},
but to keep this note short and simple we leave the details to the
reader. Furthermore, $d_{R}$ could depend on $x$ if it satisfies
certain growth condition (see below).

Every infinitesimal Hilbertian space $X$ admits a natural heat (semi)flow
$T_{t}:L^{2}\to L^{2}$ which is the gradient (semi)flow of its Cheeger
energy. Furthermore, this heat flow admits a heat kernel. In order
to get bounds on the heat kernel and Hölder regularity of functions
with $L^{p}$-Laplacian we also assume that $X$ satisfies a local
uniform Poincar\'{e} inequality. Under these assumptions Sturm \cite{Sturm1995}
showed that the following bounds hold:
\[
\frac{1}{C}\mathfrak{v}(t,x,y)e^{-\frac{d(x,y)^{2}}{C_{1}t}}\le p_{t}(x,y)\le C\mathfrak{v}(t,x,y)e^{-\frac{d(x,y)^{2}}{C_{2}t}}
\]
with 
\[
\mathfrak{v}(t,x,y):=\left(\mu(B_{\sqrt{t}}(x))\right)^{-\frac{1}{2}}\left(\mu(B_{\sqrt{t}}(y))\right)^{-\frac{1}{2}}.
\]

The following lemma will be useful to get estimates for $p_{t}(x,y)$
and $p_{lt}(x,y)$.
\begin{lem}
For any $l>0$ and any $(lt)^{2}<R$ the following is true
\[
\mathfrak{v}(t,x,y)\le\tilde{C}_{l,R}\mathfrak{v}(lt,x,y).
\]
\end{lem}
\begin{proof}
Note that by the doubling property we have 
\[
\mu(B_{\sqrt{lt}}(x))\le C_{l,R}\mu(B_{\sqrt{t}}(x))
\]
and thus 
\begin{eqnarray*}
\mathfrak{v}(t,x,y) & = & \left(\mu(B_{\sqrt{t}}(x))\mu(B_{\sqrt{t}}(y))\right)^{-\frac{1}{2}}\\
 & \le & \tilde{C}_{l,R}\left(\mu(B_{\sqrt{lt}}(x))\mu(B_{\sqrt{lt}}(y))\right)^{-\frac{1}{2}}=\tilde{C}_{l,R}\mathfrak{v}(lt,x,y).
\end{eqnarray*}
\end{proof}
\begin{cor}
\label{cor:heat bound}For any positive function $\varphi$ and $l=2C_{1}/C_{2}$
there is a constant $\hat{C}>0$ such that the following holds
\[
\int\varphi(y)p_{t}(x,y)d\mu(y)\le\hat{C}\int\varphi e^{-\frac{d(x,y)^{2}}{2C_{1}t}}p_{lt}(x,y)d\mu(y).
\]
\end{cor}
\begin{proof}
By the upper estimate for $p_{t}$ and the lower estimate for $p_{lt}$
we have 
\begin{eqnarray*}
\int\varphi(y)p_{t}(x,y)d\mu(y) & \le & C\int\varphi(y)\mathfrak{v}(t,x,y)e^{-\frac{d(x,y)^{2}}{2C_{1}t}}e^{-\frac{d(x,y)^{2}}{2C_{1}t}}d\mu(y)\\
 & \le & C\cdot\tilde{C}\int\varphi(y)e^{-\frac{d(x,y)^{2}}{2C_{1}t}}\mathfrak{v}(lt,x,y)e^{-\frac{d(x,y)^{2}}{C_{2}(lt)}}d\mu(y)\\
 & \le & C^{2}\cdot\tilde{C}\int\varphi(y)e^{-\frac{d(x,y)^{2}}{2C_{1}t}}p_{lt}(x,y)d\mu(y).
\end{eqnarray*}
\end{proof}
\begin{rem*}
Because of the term $e^{-\frac{d(x,y)^{2}}{2C_{1}t}}$, if $t$ is
chosen sufficiently small, we still get the same estimate with $\hat{C}$
depending also on $x$ if the doubling constants depend on the chosen
point but satisfy some growth condition, i.e. for each $x\in X$ there
is a $C_{R,x}$ such that 
\[
\mu(B_{2r}(x))<C_{R,x}\mu(B_{r}(x))
\]
and 
\[
C_{R,y}\le C_{R,x}e^{Dd(x,y)^{2}}.
\]

\end{rem*}
Note that in the following lemma, only Hölder continuity of functions
with $L^{p}$-Laplacian will be used further below. 
\begin{lem}
[{\cite[2.1,2.2]{Jiang2011}}]\label{lem:hoelder}Assume that $u\in W_{loc}^{1,2}(X)$
such that $\Delta u=f$ on $B_{2r}(x_{0})$ for some $f\in L_{loc}^{p}(X)$
for $p>Q\ge2$ such that $B_{2r}(x_{0})\subset\subset X$ then
\[
\sup_{B_{r}(x_{0})}|u|\le C\left\{ r^{-Q/2}\|u\|_{L^{2}(B_{2r}(x_{0}))}+r^{2-Q/p}\|f\|_{L^{p}(B_{2r}(x_{0})}\right\} .
\]
Furthermore, there is a $C'>0$ such that $u$ satisfy the above for
$B_{2R}(x_{0})$ and such that for all $0<r<R$ 
\[
\int_{B_{r}(x_{0})}|\nabla u|^{2}d\mu\le C'R^{2+Q(1-\frac{2}{p})}\|f\|_{L^{p}(B_{2R}(x_{0})}^{2}+\frac{C'}{(R-r)^{2}}\|u\|_{L^{2}(B_{2R}(x_{0}))}^{2}.
\]

In addition, $u$ is locally Hölder continuous with constants only
depending (locally) on $u$ and $f$ and the doubling constants.\end{lem}
\begin{proof}
The first two parts follow from \cite[Proposition 2.1, Proposition 2.2]{Jiang2011}
by noting that 
\[
\mu(B_{r}(x_{0}))^{-1/p}\le d_{R}^{-1/p}r^{-Q/p}.
\]
The last part follows from \cite[Theorem 5.13]{Biroli1995}.\end{proof}
\begin{rem*}
Note that by the same argument (i.e. ``lower Ahlfors regularity'')
Lemma 3.2 and 3.3 in \cite{Koskela2003} (resp. Lemma 2.4 and 2.5
in \cite{Jiang2011}) hold without changing the proofs. 
\end{rem*}
The proof of \cite[Proposition 3.4]{Koskela2003} (used also in \cite[Lemma 2.1]{Jiang2011})
only uses a very weak form of the ``upper Ahlfors regularity'' in
order to show that 
\[
r^{-1}(r\mu(B_{2r}(x)))^{\frac{1}{2}}=\left(\frac{\mu(B_{2r}(x))}{r^{\frac{1}{2}}}\right)^{\frac{1}{2}}
\]
is bounded for $r<1$ and converges to $0$. It is easy to see that
this holds if for some $\alpha>1$ and $R>0$ 
\[
\sup_{0<r<R}\frac{\mu(B_{r}(x))}{r^{\alpha}}<\infty,
\]
i.e. for some $D_{R}$ and all $0<r<R$ we have $\mu(B_{r}(x))<D_{R}r^{\alpha}$.
In the proposition below, we will show that an even weaker condition
is enough to prove the statement.

We say that $\phi:[0,T]\times X\to\mathbb{R}$ is a test function
if it is Hölder continuous, $\phi(t,\cdot)=\phi_{t}\in W^{1,2}(X),$
$(t,x)\mapsto|\nabla\phi(t,\cdot)|(x)\in L^{2}([0,T]\times X)$ and
$\phi(\cdot,x)$ is absolutely continuous on $[0,T]$ for $\mu$-almost
all $x\in X$.
\begin{prop}
Assume that for $\mu$-almost every $x$ and each $0<\alpha<1$ there
is a $C=C(x,\alpha)>0$ such that 
\[
\mu(B_{r}(x))\le C\cdot r^{\alpha}
\]
Then there exists a constant $K$ such that for every (Hölder continuous)
test function $\phi$ and almost every $x\in X$
\begin{eqnarray*}
\int_{0}^{T}\int\phi_{t}\Delta p_{t}(x,\cdot)d\mu dt & := & -\int_{0}^{T}\int\langle\nabla\phi_{t},\nabla p_{t}(x,\cdot)\rangle d\mu dt\\
 & = & \int_{0}^{T}\int\phi_{t}\partial_{t}p_{t}(x,\cdot)d\mu dt+K\phi_{0}(x).
\end{eqnarray*}
\end{prop}
\begin{rem*}
(1) A first version of this note contained a wrong proof. The usage
of the correct term and the correct adjustment is thanks to Jiang.
In \cite{Jiang2013} Jiang managed to avoid using \cite[Proposition 3.4]{Koskela2003}.

(2) The assumption holds, for example, if for $\mu$-almost every
$x$ 
\[
\dim\mu(x)=\limsup_{r\to0}\frac{\log\mu(B_{r}(x))}{\log r}\ge1,
\]
i.e. the Hausdorff dimension of $\mu$ is at least $1$. In particular,
it will hold for $RCD(K,N)$-spaces (see Lemma \ref{lem:rcd-dim}).\end{rem*}
\begin{proof}
Assume for some $0<\alpha<1$, to be chosen later on, we have 
\[
\mu(B_{r}(x))\le C\cdot r^{\alpha}
\]
for all $0<r<1$. Carefully checking the original proof, one notes
that it suffices to bound the following term for $0<r<1$ and show
that it converges to $0$ as $r$ does
\[
\begin{array}{l}
\frac{1}{r}\|\phi-\phi(0,x)\|_{L^{\infty}(B_{2r}(x)\times(0,2r))}\int_{0}^{2r}\int_{B_{2r}(x)}|\nabla p_{t}(x,\cdot)|d\mu dt\\
\le\frac{1}{r}\|\phi-\phi(0,x)\|_{L^{\infty}(B_{2r}(x)\times(0,2r))}(r\mu(B_{r}(x))^{\frac{1}{2}}\left(\int_{0}^{2r}\int_{B_{2r}(x)}|\nabla p_{t}(x,\cdot)|^{2}d\mu dt\right)^{\frac{1}{2}}\\
\le C^{\frac{1}{2}}\cdot C_{\phi}\cdot r^{-1+\beta+\nicefrac{1}{2}+\nicefrac{\alpha}{2}}\left(\int_{0}^{2r}\int_{B_{2r}(x)}|\nabla p_{t}(x,\cdot)|^{2}d\mu dt\right)^{\frac{1}{2}}
\end{array}
\]
where $\beta\in(0,1]$ is the Hölder exponent of $\phi$ and $C_{\phi}$
the Hölder constant. Choosing $\alpha\in(1-2\beta,1)$, the term is
dominated by $\tilde{C}r^{\gamma}$ for some $\gamma\in(0,1)$. In
particular, it is bounded and converges to $0$ as $r\to0$.

Furthermore, because the doubling and Poincar\'{e} constants only
depend on some $R$ we get by Sturm \cite[2.6]{Sturm1995} 
\[
|\partial_{t}p_{t}(x,y)|\le C\mathfrak{v}(t,x,y)\cdot t^{-1}e^{-\frac{d(x,y)^{2}}{C_{1}t}}.
\]

With these two facts the proof of \cite[3.4]{Koskela2003} can be
followed without any change.
\end{proof}
Finally we can state the main theorem. 
\begin{thm}
\label{thm:main}Assume the space satisfies the previous proposition
and that there is a constant $C>0$ and $T>0$ such that for every
$0<t<T$ and every $g\in W^{1,2}(X)$ the following Sobolev-Poincar\'{e}
inequality for $T_{t}$ holds 
\[
T_{t}(g^{2})(x)\le(2t+Ct^{2})T_{t}(|\nabla g|^{2})(x)+(T_{t}g(x))^{2}.
\]
Then any $u\in W_{loc}^{1,2}(X)$ with $\Delta u_{|\Omega}\in L_{loc}^{p}(\Omega)$
for some open $\Omega\subset X$ is locally Lipschitz in $\Omega$
if $p>Q\ge2$. \end{thm}
\begin{rem*}
1) Even though we made the assumption $Q\ge2$, similar to \cite{Jiang2011},
it is possible to show the same for $Q\in[1,2)$.

2) If the Sobolev-Poincar\'{e} inequality holds uniformly, i.e. $C$
does not depend on $x$ then $T_{t}$ satisfies the Bakry-\'Emery
condition and vice versa (see \cite{Bakry1997}). In the proof the
condition is only required to hold for $t$ sufficiently small and
for functions with support in a neighborhood of $x$, thus it might
be interpreted as a local curvature condition.\end{rem*}
\begin{proof}
Note that Ahlfors regularity is used in the proof of \cite[3.1, 3.2., 3.3]{Jiang2011}
only three times, namely Inequality $(3.4)$ on page 291, on the bottom
of page 294, and for Inequality $(3.12)$. We will only show how to
adjust these steps and leave out the details of the remaining parts.

Inequality $(3.4)$ on page $291$ of \cite{Jiang2011} can be proven
as follows:
\begin{eqnarray*}
|w(t,x)| & = & |u(x)\phi(x)-T_{t}(u\phi)(x_{0})|\\
 & = & |u(x)\phi(x)-u(x_{0})\phi(x_{0})+u(x_{0})\phi(x_{0})-T_{t}(u\phi)(x_{0})|\\
 & \le & C\cdot C(u,f)d(x,x_{0})^{\delta}+\int_{X}|u(x_{0})\phi(x_{0})-u(y)\phi(y)|p_{t}(x_{0},y)d\mu(y)\\
 & \le & C\cdot C(u,f)\big\{ d(x,x_{0})^{\delta}\\
 &  & +\int|u(x_{0})\phi(x_{0})-u(y)\phi(y)|e^{-\frac{d(x_{0},y)^{2}}{2C_{\text{1}}t}}p_{lt}(x_{0},y)d\mu(y)\big\}
\end{eqnarray*}
where we applied Corollary \ref{cor:heat bound}. Splitting the integral
as in \cite[p. 163]{Koskela2003} we get 
\begin{eqnarray*}
|w(t,x)| & \le & C\cdot C(u,f)\big\{ d(x,x_{0})^{\delta}\\
 &  & +\int_{B_{t^{\frac{1}{3}}}(x_{0})}d(x_{0},y)^{\delta}p_{lt}(x_{0},y)d\mu(y)\\
 &  & +2\|u\|_{\infty}e^{-\frac{t^{\frac{2}{3}}}{2C_{1}t}}\int_{X\backslash B_{t^{\frac{1}{3}}}(x_{0})}p_{lt}(x_{0},y)d\mu(y)\big\}\\
 & \le & C\cdot C(u,f)(d(x,x_{0})^{\delta}+t^{\delta/2}\}.
\end{eqnarray*}

Using the same argument we can derive the inequality on the bottom
of page 294 in \cite{Jiang2011}
\begin{eqnarray*}
\int_{X}|Dw(t,x)|^{2}p_{t}(x_{0},x)d\mu(x) & < & \frac{1}{2t}\int_{X}w^{2}(t,x)p_{t}(x_{0},x)d\mu(x)\\
 & \le & C\cdot C(u,f)^{2}\frac{1}{2t}\int_{X}(d(x,x_{0})^{\delta}+t^{\delta/2})^{2}p_{lt}(x_{0},x)d\mu(xy)\\
 & \le & C\cdot C(u,f)^{2}t^{\delta-1}.
\end{eqnarray*}

Finally, to get the lower bound on $\int_{0}^{T}\frac{d}{dt}J(t)dt$
we only have adjust inequality $(3.12)$ on page 295 of \cite{Jiang2011}.
The only term where Ahlfors regularity was used is the following 
\[
\int_{0}^{T}t^{-\epsilon}\left(\int p_{t}(x_{0},x)^{\frac{p}{p-2}}d\mu(x)\right)^{1-\frac{2}{p}}dt,
\]
where $\epsilon>0$ is chosen such that $\epsilon+\frac{Q}{p}<1$.

First note that $\frac{p}{p-2}=1+\frac{2}{p-2}$ and 
\[
p_{t}(x_{0,}x)^{\frac{2}{p-2}}\le(C\mu(B_{\sqrt{t}}(x_{0})^{-\frac{1}{2}}\mu(B_{\sqrt{t}}(x)^{-\frac{1}{2}})^{\frac{2}{p-2}}\le\tilde{C}t^{-\frac{Q}{p-2}}.
\]
Therefore,
\begin{eqnarray*}
\int_{0}^{T}t^{-\epsilon}\left(\int p_{t}(x_{0},x)^{\frac{p}{p-2}}d\mu(x)\right)^{1-\frac{2}{p}}dt & \le & \int_{0}^{T}t^{-\epsilon}\left(\tilde{C}t^{-\frac{Q}{p-2}}\int p_{t}(x_{0},x)d\mu(x)\right)^{\frac{p-2}{p}}dt\\
 & \le & \hat{C}\int_{0}^{T}t^{-\epsilon-\frac{Q}{p}}\left(\int p_{t}(x_{0},x)d\mu(x)\right)^{\frac{p-2}{p}}dt\\
 & \le & C'<\infty,
\end{eqnarray*}
where $C'$ depends only on $\hat{C}$, $T$ and $\epsilon+\frac{Q}{p}<1$,
which can be chosen uniformly in a neighborhood of $x_{0}$.

The remaining parts of the proof follow by directly copying Jiang's
proof of \cite[3.1,3.2, 3.3]{Jiang2011}.\end{proof}
\begin{cor}
Assume $X$ satisfies the $RCD(K,N)$ condition. Then any $u\in W_{loc}^{1,2}(X)$
with $\Delta u\in L_{loc}^{p}$ and $p>N$ is locally Lipschitz continuous.
\end{cor}
The corollary is a result of the following lemma which might be useful
in its own right. First let us state the Bishop-Gromov volume comparison
inequality: $X$ is said to satisfy $BG(K,N)$ if for each $0<r<R<\infty$
and $x\in X$ 
\[
\frac{\mu(B_{R}(x))}{\mu(B_{r}(x))}\le\frac{V_{K,N}(R)}{V_{K,N}(r)},
\]
where $V_{K,N}(r)$ is the volume of the ball of radius $r$ in the
$N$-dimensional model space $\mathbb{M}_{K,N}$ of constant curvature
$K$. Note that $V_{K,N}$ is non-decreasing and locally Lipschitz
in a neighborhood of every $r>0$. 
\begin{lem}
\label{lem:rcd-dim}Assume $X$ satisfies $BG(K,N)$. Then for every
$x\in X$ and $R_{0}>0$ there is a $C>0$ such that for every $0<r<R_{0}$
\[
\mu(B_{r}(x))\le C\cdot r.
\]
In particular, $C$ can be chosen uniformly in a neighborhood of $x$.\end{lem}
\begin{rem*}
This result is based on a proof of Lipschitz continuity of $x\mapsto\mu(B_{r}(x))$
by Ba\v{c}ák-Hua-Jost-Kell \cite{BHJK2013} based on Buckley's $\delta$-annular
decay property in \cite{Buckley1999}. Earlier, but independently,
Kitabeppu \cite[Lemma 3.1]{Kitabeppu2013} discovered this fact as
well.\end{rem*}
\begin{proof}
Take any $y\in X\backslash\{x\}$. Then for $0<\epsilon<r=d(x,y)$
\[
B_{\epsilon}(x)\subset B_{r+\epsilon}(y)\backslash B_{r-\epsilon}(y).
\]
Thus, by continuity of $\epsilon\mapsto\mu(B_{\epsilon}(x))$, it
suffices to show that 
\[
\frac{\mu(B_{r+\epsilon}(y)\backslash B_{r-\epsilon}(y))}{2\epsilon}
\]
is bounded for $0<\epsilon\ll1$. 

By the $BG(K,N)$ condition
\begin{eqnarray*}
\frac{\mu(B_{r+\epsilon}(y)\backslash B_{r-\epsilon}(y))}{\mu(B_{r+\epsilon}(y))} & = & \frac{\mu(B_{r+\epsilon}(y))-\mu(B_{r-\epsilon}(y))}{\mu(B_{r+\epsilon}(y))}=1-\frac{\mu(B_{r-\epsilon}(y))}{\mu(B_{r+\epsilon}(y))}\\
 & \le & 1-\frac{V_{K,N}(r-\epsilon)}{V_{K,N}(r+\epsilon)}=\frac{V_{K,N}(r+\epsilon)-V_{K,N}(r-\epsilon)}{V_{K,N}(r+\epsilon)}.
\end{eqnarray*}
Because $V_{K,N}$ is locally Lipschitz, i.e. $L=\operatorname{Lip}V_{K,N}|_{(r-\epsilon_{0},r+\epsilon_{0})}<\infty$
for some $\epsilon_{0}>0$, 
\[
\frac{\mu(B_{r+\epsilon}(y)\backslash B_{r-\epsilon}(y))}{2\epsilon}\le\frac{L\cdot\mu(B_{r+\epsilon_{0}}(x))}{V_{K,N}(r)}\le M<\infty
\]
for all $0<\epsilon<\epsilon_{0}$. Also, note that $M$ can be chosen
uniformly in a neighborhood of $x$.
\end{proof}
\bibliographystyle{amsalpha}
\bibliography{bib}

\end{document}